\documentclass[12pt, reqno, twoside, letterpaper]{amsart}

\usepackage{paperstyle,bm}
\usepackage{color}
\usepackage[usenames,dvipsnames,svgnames,table]{xcolor}
\usepackage{paperstyle}
\usepackage{graphicx}
\usepackage{mathtools}
\usepackage{todonotes}
\usepackage[norefs,nocites]{refcheck}

\def\mbold{{\mathbf{m}}}

\def\({\left(}
\def\){\right)}
\def\fl#1{\left\lfloor#1\right\rfloor}

\newcommand{\be}{\begin{equation}}
\newcommand{\ee}{\end{equation}}

\def\balpha{{\bm\alpha}}

\def\bnu{{\bm\nu}}
\def\Nab{N_{\balpha,\mbold}}
\def\RQ{{\R\setminus\Q}}
\def\vh{\bm h}
\newcommand{\flo}[1]{\lfloor #1 \rfloor}
\newcommand{\nearint}[1]{\left\llbracket #1 \right\rrbracket}

\title[Greatest common divisor of integer parts of polynomials]
{On the greatest common divisor of\\ integer parts of polynomials}

\author[W.~D.~Banks]{William Banks}

\address{Department of Mathematics, 
 University of Missouri, 
 Columbia MO 65211, USA}

\email{bankswd@missouri.edu}

\author[I. E. Shparlinski] {Igor E. Shparlinski}
 
\address{Department of Pure Mathematics,
 University of New South Wales,
 Sydney, NSW 2052, Australia}
 
\email{igor.shparlinski@unsw.edu.au}

\date{\today}

\begin{document}

\begin{abstract} Motivated by a question of V.~Bergelson and F.~K.~Richter (2017), we obtain asymptotic formulas for the number of relatively 
prime tuples composed of positive integers $n\le N$ and integer parts of 
polynomials evaluated at $n$. The error terms in our formulas 
are of various strengths depending on the 
Diophantine properties of the leading coefficients of these polynomials. 
\end{abstract}
 

\maketitle

\vskip0.5in

\section{Introduction}

\subsection{Motivation} 
Let $\fl{t}$ and $\{t\}$ denote the integer and fractional parts of
a real number $t$, respectively; thus, $t=\fl{t}+\{t\}$
for all $t\in\R$.

Watson~\cite{Wat}, answering a question of K.~F.~Roth, proved that 
for any given irrational number $\alpha$,
the set of positive integers $n$ for which
$\gcd(n,\fl{\alpha n})=1$ has a natural density
\be\label{eq:Wat}
\delta\(\{n\in\N:~\gcd(n,\fl{\alpha n})=1\}\)=\frac{6}{\pi^2}.
\ee
In the same paper, Watson showed that a similar
result holds for all rational numbers $\alpha$, albeit with a
natural density that depends on $\alpha$ and differs from $6/\pi^2$. 
Shortly thereafter,
Estermann~\cite{Ester} gave a different proof
of a slight generalization of Watson's theorem.
Later, Erd\H os and Lorentz~\cite{ErdLor} gave
sufficient conditions for a differentiable function
$f:[1,\infty)\to\R$ to satisfy
$$
\delta\(\{n\in\N:~\gcd(n,\fl{f(n)})=1\}\)=\frac{6}{\pi^2}.
$$
The problem of finding functions $f$ with this property has been studied by
several authors; see~\cite{BergRich} for
a historical account of these results.

The present paper is inspired by a result of Bergelson and 
Richter~\cite{BergRich}, which asserts that the natural density
$$
\delta(\{n\in\N:~\gcd(n,\fl{f_1(n)},\ldots,\fl{f_k(n)})=1\})
=\frac{1}{\zeta(k+1)}
$$ 
for any functions $f_1,\ldots,f_k$ belonging to a given
Hardy field $\cH$ and satisfying some mild conditions. Here,
$\zeta(s)$ is the Riemann zeta function. At the end
of their paper as a natural
extension to Watson's original result~\eqref{eq:Wat}, 
Bergelson and Richter pose the following
question (see~\cite[Question~1]{BergRich}):

\begin{question} 
\label{quest:gcd}
Let $\alpha_1,\ldots,\alpha_k$ be irrational numbers.
Is it true that the natural density of the set
\be\label{eq:caesar}
\{n\in\N:~\gcd(n,\fl{\alpha_1n},\ldots,\fl{\alpha_k n^k})=1\}
\ee
exists and is equal to $1/\zeta(k+1)$? 
\end{question}

 In this paper, we show this question has an affirmative answer
whenever the numbers $\{\alpha_j\}$ satisfy some mild Diophantine conditions.
For example, when all of the numbers $\{\alpha_j\}$ are of {\it finite type\/},
we establish an asymptotic formula with a strong bound on the error term.
Our techniques also apply to certain classes of \emph{Liouville numbers} 
(i.e., numbers of infinite type) but with
somewhat weaker bounds on the error terms. We note that, in complete
generality, the original question remains open.

To formulate our various results precisely,
we first recall some standard notions from the theory of Diophantine 
approximations.

\subsection{Types of irrational numbers}
Let $\nearint{t}$ denote the
distance from a real number $t$ to the nearest integer:
\be\label{eq:near-int}
\nearint{t}\defeq \min\limits_{n\in\Z}|t-n|\qquad(t\in\R).
\ee
For any irrational number $\alpha$, we define its \emph{type}
$\tau$ by the relation
$$
\tau\defeq\sup\left\{\vartheta\in\R:
\mathop{\underline{\rm lim}}\limits_{q\in\N}
~q^\vartheta\nearint{q\alpha}=0\right\}.
$$
We say that $\alpha$ is of \emph{finite type}
when $\tau<\infty$. Using Dirichlet's approximation theorem,
one sees that $\tau\ge 1$ for every irrational $\alpha$.
The celebrated theorems of Khinchin~\cite{Khin} and of 
Roth~\cite{Roth1, Roth2}
assert that $\tau=1$ for almost all real (in the sense of the
Lebesgue measure) and all irrational algebraic numbers $\alpha$,
respectively. 

Similarly, for any $\alpha\in\RQ$ we define its 
\emph{exponential type} $\tau_\star$ by
$$
\tau_\star\defeq \sup\left\{\vartheta\in\R:~
\mathop{\underline{\rm lim}}\limits_{q\in\N}
~\exp(q^\vartheta)\nearint{\alpha q}=0\right\}.
$$
We say that $\alpha$ is of \emph{finite exponential type}
whenever $\tau_\star<\infty$. Note that if $\alpha$ is of 
finite type $\tau$, then its exponential type $\tau_\star$ is
also finite, and one has $\tau_\star\le\tau$. The converse
is not true in general.

\subsection{Statement of results}
Let $k\ge 1$ be a fixed integer. Given a sequence
\be\label{eq:alpha}
\balpha\defeq (\alpha_j)_{j=1}^k
\ee 
of irrational real numbers and a sequence
$\mbold\defeq (m_j)_{j=1}^k$ of integers such that
$$
1=m_1<m_2<\cdots<m_k,
$$
denote by $\Nab(x)$ the number of positive integers $n\le x$
that satisfy
\be\label{eq:gcd-condn}
\gcd(n,\flo{\alpha_1n^{m_1}},\flo{\alpha_2n^{m_2}},
\ldots,\flo{\alpha_kn^{m_k}})=1.
\ee

\newpage

\begin{theorem}\label{thm:finite-type}
Let $\balpha$  as in~\eqref{eq:alpha} be such that every $\alpha_j$ is an
irrational number of finite type not exceeding $\tau$. Then
the estimate 
$$
\Nab(x)=\frac{x}{\zeta(k+1)}+O\(x^{1-\gamma+o(1)}\)
\qquad(x\to\infty),
$$
holds with
$$
\gamma\defeq\begin{cases}
(3\tau+2)^{-1}&\quad\hbox{if $k=1$},\\
\frac{1}{8}\min\bigl\{(m_k\tau)^{-1},(m_k^2-m_k)^{-1}\bigr\}
&\quad\hbox{if $k\ge 2$}.\\
\end{cases}
$$
\end{theorem}

A proof of Theorem~\ref{thm:finite-type} is given in~\S\ref{sec:proof-thm1}.
As alluded to above, for almost all vectors $\balpha$
(in the sense of Lebesgue measure) one can use $\tau\defeq 1$
in applications of Theorem~\ref{thm:finite-type}; thus, for such
vectors one can take 
$$\gamma\defeq \tfrac{1}{8}(m_k^2-m_k)^{-1}
$$ 
in all these cases.

We note that although we have optimized the general shape of the dependence 
on $m_k$ and $\tau$, the constant $\frac18$ in the above is certainly not optimal 
and, with a bit of tedious work, can be improved. 
 
Our next  result tests the limits of our approach
as we consider abnormally well-approximable vectors $\balpha$.
More precisely, in some cases in which the terms of the sequence
$\balpha$ as in~\eqref{eq:alpha}  are all irrational and of finite
\emph{exponential type}, we still manage to establish the expected
asymptotic relation, albeit with a weaker error term.

\begin{theorem}\label{thm:fin-exp-type}
Let $\balpha$  as in~\eqref{eq:alpha}  be  
such that every $\alpha_j$ is an irrational number of finite
exponential type not exceeding $\tau_\star$. Then
the estimate 
$$
\Nab(x)=\frac{x}{\zeta(k+1)}+O\(x^{1-\gamma_\star+o(1)}\)
\qquad(x\to\infty),
$$
holds with
$$
\gamma_\star\defeq\begin{cases}
\min\left\{\tau_\star^{-1},\tfrac12(\tau_\star^{-1}+1)\right\}&\quad\hbox{if $k=1$},\\
\dfrac{1-(m_k^2-m_k+1)\tau_\star}{(m_k^2+2)\tau_\star}
&\quad\hbox{if $k\ge 2$}.
\end{cases}
$$
%%
%%
%%\begin{itemize}
%%\item If $k=1$, then the estimate
%%$$
%%\Nab(x)=\frac{x}{\zeta(2)}
%%+O\(x(\log x)^{-\gamma_\star+o(1)}\)\qquad(x\to\infty)
%%$$
%%holds with
%%$$
%%\gamma_\star\defeq\min\{\tau_\star^{-1},\tfrac12(\tau_\star^{-1}+1)\}.
%%$$ 
%%\item If $k\ge 2$, then the estimate
%%$$
%%\Nab(x)=\frac{x}{\zeta(k+1)}
%%+O\(x(\log x)^{-\gamma_\star+o(1)}\)\qquad(x\to\infty)
%%$$
%%holds with
%%$$
%%\gamma_\star\defeq\frac{1-(m_k^2-m_k+1)\tau_\star}{(m_k^2+2)\tau_\star}.
%%$$ 
%%\end{itemize}
\end{theorem}

A proof of Theorem~\ref{thm:fin-exp-type} is given  in~\S\ref{sec:proof-thm2}.

\begin{remark}Examining our proofs, one can immediately notice that 
without changing anything in the statements of Theorems~\ref{thm:finite-type} and \ref{thm:fin-exp-type},
one can replace $\alpha_j n^{m_j}$, $ j =2, \ldots, k$,
in~\eqref{eq:caesar} with   $\alpha_j n^{m_j} + g_j(n)$ where $g_j \in \R[X]$, 
$\deg g_j < m_j$ (however we still have to keep $\alpha_1 n$ in~\eqref{eq:caesar}). 
\end{remark}

\smallskip\section{Preliminaries}

\subsection{Denominators of Diophantine approximations}
The following simple result gives bounds on the denominators 
of certain rational approximations to an irrational number of finite type.

\begin{lemma}\label{lem:Dioph-one}
Suppose that $\alpha\in\RQ$ has finite type $\tau$,
and that $\varpi\in(0,\tau^{-1})$. If $Q$ is large enough
$($depending on $\alpha$ and $\varpi)$,
then there are integers $a$ and $q$ such that
\be\label{eq:Diophantine1}
\biggl|\alpha-\frac{a}{q}\biggr|<\frac{1}{qQ},
\qquad\gcd(a,q)=1,\qquad Q^\varpi<q\le Q.
\ee
\end{lemma}

\begin{proof}
By Dirichlet's approximation theorem,
there are coprime integers $a$ and $q\le Q$ such that the
first inequality of~\eqref{eq:Diophantine1} holds. Then
$$
\nearint{\alpha q}\le |\alpha q-a|<Q^{-1}.
$$
On the other hand, since $\alpha$ is of type $\tau<\varpi^{-1}$,
we have
$$
q^{1/\varpi}\nearint{\alpha q}\ge 1
$$
if $q$ is large enough. Combining these inequalities,
the lemma follows.
\end{proof}

We also use a similar result for irrational numbers of finite
exponential type; the proof is nearly identical to
that of Lemma~\ref{lem:Dioph-one}.

\begin{lemma}\label{lem:Dioph-two}
Suppose that $\alpha\in\RQ$ has finite exponential
type $\tau_\star$, and that $\varpi\in(0,\tau_\star^{-1}-1)$.
If $Q$ is sufficiently large $($depending on $\alpha$ and $\varpi)$,
then there are integers $a$ and~$q$ such that
$$
\biggl|\alpha-\frac{a}{q}\biggr|<\frac{1}{qQ},
\qquad\gcd(a,q)=1,\qquad (\log Q)^{\varpi+1}<q\le Q.
$$ 
\end{lemma}

\subsection{Discrepancy and the Koksma-Sz\"usz inequality}
\label{sec:discrepancy}

Let us consider the collection $\cS$ consisting of all subsets
$S$ of $\Omega\defeq [0,1)^k$ of the form
$$
S=\bigotimes_{1\le j\le k}[a_j,b_j) 
$$
with $0\le a_j<b_j\le 1$ for each $j$. For any given sequence
$\bm v\defeq (\bm v_n)_{n\ge 1}$ of vectors $\bm v_n\in\Omega$
and a fixed set $S\in\cS$, we denote
$$
A(\bm v,S;N)\defeq \bigl|\{n\le N:~\bm v_n\in S\}\bigr|.
$$
The (extreme) discrepancy is the quantity defined by
$$
\sD(\bm v;N)\defeq \sup\limits_{S\in\cS}
\left|\frac{A(\bm v,S;N)}{N}-m(S)\right|
\qquad\text{with}\quad
m(S)\defeq \prod_{1\le j\le k}(b_j-a_j).
$$
Note that if the vectors $\bm v_n$ in $\bm v$ are chosen uniformly 
at random from $\Omega$ and independently for each $n$,
then $m(S)$ (the Lebesgue measure of the
subset $S\subset\Omega$) represents the proportion of the vectors
$\bm v_n$  expected to lie in $S$.

One of the basic tools used to study uniformity of
distribution is the celebrated
\emph{Koksma--Sz\"usz inequality\/}~\cite{Kok,Sz} 
(see also Drmota and Tichy~\cite[Theorem~1.21]{DrTi}),
which links the discrepancy of a sequence of points to certain
exponential sums.
To formulate the result, let us recall the standard notation. 
$$
\e(t)\defeq \exp(2\pi it)\qquad(t\in\R).
$$ 
Also, identifying each real sequence $\bm{v}\defeq (v_j)_{j=1}^k$ with
the vector $(v_1,\ldots,v_k)\in\R^k$, we denote the inner
 product of two such sequences $\bm{v}$ and $\bm{w}$ by
$$
\langle\bm{v},\bm{w}\rangle\defeq \sum_{j=1}^kv_jw_j.
$$
The Koksma--Sz\"usz inequality can be stated as follows.

\begin{lemma}\label{lem:K-S}
There is an absolute constant $C>0$ with the following property.
For any integer $H > 1$ and any sequence 
of points $\bm v\defeq (\bm v_n)_{n\ge 1}$ in $\Omega$, we have
$$
\sD(\bm v;N)\le C^k\(\frac{1}{H}
+\frac1N\sum_{0<\|\vh\|\le H}\frac{1}{r(\vh)}
\left|\sum_{n=1}^N\e(\langle\vh,\bm v_n\rangle)\right|\),
$$
where 
$$
\|\vh\|\defeq \max_j|h_j|, \qquad
r(\vh) \defeq \prod_{j=1}^k \max\(|h_j|,1\),
$$
and the sum is taken over all vectors
$\vh= (h_1,\ldots,h_k)\in\Z^k$ with $0<\|\vh\|\le H$.
\end{lemma}

\subsection{Bounds on Weyl sums}
\label{sec:bounding-Weyl-sums}

We use the following result of Shparlinski and
Thuswaldner~\cite[Lemma~3.2]{ShTh}. 

\begin{lemma}\label{lem:Weyl-h}
Let $m \ge 2$ be a fixed integer. 
Suppose that $\alpha\in[0,1)$ satisfies
$$
\left|\alpha-\frac{a}{q}\right|<\frac{1}{q^2}
$$
with some coprime integers $a$ and $q\ge 1$. 
Then, for any integer $h \ne 0$, any polynomial $g(X) \in \R[X]$ of degree 
at most $m-1$, and any integer $N$, we have 
\be
\label{eq:Bound little o} 
\sum_{n=1}^N \e(h\alpha n^m+g(n))
\ll N^{1+o(1)} \Delta^{1/(m^2-m)}
\qquad(N\to\infty), 
\ee
and 
\be
\label{eq:Bound log} 
\sum_{n=1}^N \e(h\alpha n^m+g(n))\ll N(\log N)\Delta^{1/(m^2-m +2)}, 
\ee
where
$$
\Delta\defeq q^{-1}|h|+N^{-1}+qN^{-m}+\gcd(q,h)N^{-m+1}.
$$ 
\end{lemma} 

At first glance, the statement of Lemma~\ref{lem:Weyl-h} may appear to be
different from~\cite[Lemma~3.2]{ShTh}, which is formulated for Weyl sums
with polynomials of the form $h f(X)$ with $f$ a real polynomial of 
degree $m\ge 2$. However, since the bound given in~\cite{ShTh}
depends only on the leading term of $f$, Lemma~\ref{lem:Weyl-h} is
actually equivalent, for it corresponds to the choice
$f(X)\defeq \alpha X^m+h^{-1}g(X)$. 
We also remark (as in~\cite{ShTh}) that the bound~\eqref{eq:Bound little o}
has a smaller exponent of $\Delta$ than that of~\eqref{eq:Bound log},
hence the first bound is typically stronger. For very small $q$, however,
the factor $N^{o(1)}$ can make~\eqref{eq:Bound little o} trivial, 
whereas~\eqref{eq:Bound log} is nontrivial in the same situation.

For the case $m=2$, we have a more precise statement,
which follows from the Weyl differencing method;
see~\cite[Equation~(3.5)]{Bak-DI}. 

\begin{lemma}\label{lem:QuadrSum-h}
For any integer $h \ne 0$, any linear polynomial $g(X) \in \R[X]$,
and any integer $N$, we have 
$$
\left|\sum_{n=1}^N\e(h\alpha n^2+g(n))\right|^2
\ll\sum_{v=1}^N\min\(N,\frac{1}{\nearint{2hv\alpha}}\), 
$$
where $\nearint{\cdot}$ is defined by~\eqref{eq:near-int}.
\end{lemma} 

We also need a version of Lemma~\ref{lem:Weyl-h} to handle the case $m=1$,
i.e., the case of linear sums.

\begin{lemma}
\label{lem:LinSum-h} 
Suppose that $\alpha\in[0,1)$ satisfies
$$
\left|\alpha-\frac{a}{q}\right|<\frac{1}{q^2}
$$
with some coprime integers $a$ and $q\ge 1$. 
Then, for any integer $h \ne 0$, we have 
$$
\sum_{n=1}^N \e\(h\alpha n\) \ll N \Delta, 
$$
where $\Delta\defeq q^{-1}|h|+qN^{-1}$.
\end{lemma} 

\begin{proof} For $|h|\ge\frac12q$ the bound is trivial, thus we can
assume $|h|<\frac12q$.

Using the well known inequality (see, e.g.,~\cite[Equation~(8.6)]{IwKow})
$$
\left|\sum_{n=1}^N\e(h\alpha n)\right|
\le\min\(N,\frac{1}{2\llbracket h\alpha\rrbracket}\).
$$
Since $\gcd(a,q)=1$ and $0<|h|<\frac12q$,
the ratio $ah/q$ is a non-integer rational number whose
denominator (when it is expressed in reduced form)
is at most $q$; therefore, $\nearint{ah/q}\ge q^{-1}$. Since 
$|h/q^2|\le (2q)^{-1}$, we conclude that
$\nearint{h\alpha}\ge(2q)^{-1}$, and the lemma follows.
\end{proof} 

\subsection{Bounds on some reciprocal sums}

The following well-known result is used in conjunction with
Lemma~\ref{lem:QuadrSum-h}; see, e.g.,~\cite[Lemma~3.2]{Bak-DI}. 

\begin{lemma}\label{lem:RecipSum}
Suppose that $\alpha\in[0,1)$ satisfies
$$
\left|\alpha-\frac{a}{q}\right|<\frac{1}{q^2}
$$
with some coprime integers $a$ and $q\ge 1$. 
For any real integer $K, N\ge 1$, we have 
$$
\sum_{\nu=1}^K\min\(N,\frac{1}{\llbracket \nu\alpha\rrbracket}\)
\ll(N+q\log q)(K/q+1). 
$$
\end{lemma}

\subsection{Some elementary calculus}
We also need the following straightforward statements.

\begin{lemma}
\label{lem:elem} 
For any real numbers $u>0$ and $v\ge 1$, the sequence 
$$
\(\frac{u}{v^{m-1}}\)^{1/(m^2-m)}\qquad\text{with}\quad m=2,3,4,\ldots
$$
is nondecreasing if $u \le v$. 
\end{lemma}

\begin{lemma}
\label{lem:elem TWO} 
For any real numbers $u>0$ and $v\ge 1$, the sequence 
$$
\(\frac{v^m}{u}\)^{1/(m^2-m+2)}
\qquad\text{with}\quad m=2,3,4,\ldots,M
$$
is nondecreasing if $v\le u^{1/M}$. 
\end{lemma}

\section{Proof of Theorem~\ref{thm:finite-type}}
\label{sec:proof-thm1}

\subsection{Preliminary transformation and plan of the proof}
\label{sec:plan}

Our approach is based on the following
equivalence, which is easily verified:
\be\label{eq:characterize}
\fl{t}\equiv 0\bmod d
\quad\Longleftrightarrow\quad
\{t/d\}\in[0,d^{-1})\qquad(t\in\R,~d\in\N).
\ee

We begin our estimation of $\Nab(x)$ by applying a familiar
inclusion-exclusion argument, using the M\"obius function to
detect the coprimality condition~\eqref{eq:gcd-condn}:
\begin{align*} 
\Nab(x)&=\sum_{n\le x}
\ssum{d\,\mid\,\gcd(n,\flo{\alpha_1n^{m_1}},
\ldots,\flo{\alpha_kn^{m_k}})}\mu(d)
=\sum_{d\le x}\mu(d)
\ssum{n\le x/d\\
\flo{\alpha_jd^{m_j}n^{m_j}}
\equiv 0\bmod d~\forall j} 1.
\end{align*}
Using the criterion~\eqref{eq:characterize} it follows that
\be\label{eq:birdsong}
\Nab(x)
=\sum_{d\le x}\mu(d)\cdot
\bigl|\{n\le x/d:~\bnu_{d,n}\in\Omega_d\}\bigr|,
\ee
where $\Omega_d$ is used to denote the subset $[0,d^{-1})^k$ of $\R^k$,
and $\bnu_d\defeq (\bnu_{d,n})_{n\ge 1}$ is the sequence of vectors in 
$[0,1)^k$ given by
$$
\bnu_{d,n}\defeq\bigl(\{\alpha_1d^{m_1-1}n^{m_1}\},
\ldots,\{\alpha_kd^{m_k-1}n^{m_k}\}\bigr).
$$
The strength of our estimate for $\Nab(x)$ via~\eqref{eq:birdsong}
depends to a large extent on the Diophantine properties of the sequence
$\balpha\defeq (\alpha_j)_{j=1}^k$.

The plan of the proof is as follows:
\begin{itemize}
\item[$(1)$]  For a real parameter $D\in(1,x]$, we split the sum in~\eqref{eq:birdsong} 
into two sums, one varying over large $d$ (i.e., $d>D$), the other over small $d$
(i.e., $d\le D$).
\item[$(2)$] For the sum over large $d$, we obtain only an upper bound
using the trivial bound
$$
\ssum{n\le x/d\\
\flo{\alpha_jd^{m_j}n^{m_j}}
\equiv 0\bmod d~\forall j} 1 \le \ssum{n\le x/d\\
\flo{\alpha_1 d n}
\equiv 0\bmod d} 1
$$
(which holds since $m_1=1$)
along with some ideas from~\cite{Wat}.
\item[$(3)$] For the sum over small $d$, we require an asymptotic formula.
To derive such a formula, we relate the conditions 
$\flo{\alpha_jd^{m_j}n^{m_j}}\equiv 0\bmod d$ for all $j$
to a certain uniformity of distribution problem, where we can then 
apply modern bounds on Weyl sums.
\end{itemize}

\subsection{Large $d$}
\label{sec:Large d ONE}

First, we consider the ``tail'' contribution to~\eqref{eq:birdsong} coming from integers $d>D$, where $D$
is a real parameter to be specified later.
We follow some ideas of Watson~\cite{Wat}.

Since $m_1=1$,  we have the trivial bound:
$$
\bigl|\{n\le x/d:~\bnu_{d,n}\in\Omega_d\}\bigr|\le T(x,d),
$$
where
\begin{align*}
%%\be
%%\begin{split}
%%\label{eq:Def T}
T(x,d)& \defeq  \bigl|\{n\le x:~n \equiv \flo{\alpha_1n}
\equiv 0\bmod d\}\bigr|\\
&\,  = \bigl|\{n\le x/d:~\flo{\alpha_1dn}
\equiv 0\bmod d\}\bigr|.
%%\end{split}
%%\ee  %% \commI{Numbered, added the 1st line}
\end{align*}
We can assume that $d\le x$, for otherwise, $T(x,d)=0$.
Let $n\le x/d$ be fixed, and observe that  the congruence
$\flo{\alpha_1dn}\equiv 0\bmod d$ is equivalent to the fact that
$\flo{\alpha_1dn}=dm$ with some integer $m$, hence
$$
%%\be\label{eq:alpha1n}
\alpha_1n=m+\frac{\{\alpha_1dn\}}d.
%%\ee
$$

 For any fixed $\varpi\in(0,\tau^{-1})$, Lemma~\ref{lem:Dioph-one}
shows that for all large $Q$ (depending on $\alpha_1$ and $\varpi$)
there are integers $a$ and $q$ such that
\be\label{eq:a/q-alpha_1}
\biggl|\alpha_1-\frac{a}{q}\biggr|<\frac{1}{qQ},
\qquad\gcd(a,q)=1,\qquad Q^\varpi<q\le Q.
\ee 
It is convenient to assume that $Q\le x^{O(1)}$. This condition is
not restrictive as it holds for our choice of parameters at the
optimization stage.
Using~\eqref{eq:a/q-alpha_1} and the fact that $n\le x/d$, the inequality
\be
\label{eq:DiophApprox2}
\left|\frac{a n}{q}-m\right|
 \le |\alpha_1n-m|+n\left|\alpha_1-\frac{a}{q}\right|
<\frac{1}{d}+\frac{x}{dqQ}
\ee
holds with some coprime integers $a$ and $q$ such that
$Q^\varpi<q\le Q$.

If both inequalities
\be
\label{eq:ineqsforlarge-d} 
d \ge 2q \mand d\ge \frac{2x}{Q}
\ee
hold, then~\eqref{eq:DiophApprox2} implies that
$a n/q=m\in\N$, hence $q\mid n$. In this case, there are
at most $x/(dq)$ such positive integers $n \le x/d$, and so 
\be
\label{eq:large d} 
T(x,d)\ll\frac{x}{dq}<\frac{x}{dQ^\varpi}.
\ee

On the other hand, if either inequality
in~\eqref{eq:ineqsforlarge-d} fails, then 
$$
d\ll q+\frac{x}{Q}.
$$
In this case,~\eqref{eq:DiophApprox2} implies
$$
an=mq+O\bigl(q/d+x/(dQ)\bigr).
$$
Since $\gcd(a,q)=1$, it follows that $n$ belongs to one of
$O\(q/d+x/(dQ)\)$ distinct residue classes modulo $q$.
Since each residue class modulo $q$ contains no more than
$O\(x/(dq)+1\)$ positive integers $n \le x/d$, we get that
\begin{align*}
T(x,d)&\ll\(q/d+x/(dQ)\)\(x/(dq)+1\)\\
&=x/d^2+x^2/(d^2qQ)+q/d+x/(dQ)\\
&\ll x/d^2+x^2/(d^2Q^{1+\varpi})+Q/d+x/(dQ),
\end{align*}
where we have used~\eqref{eq:a/q-alpha_j} in the last step.
This implies the slightly weaker bound
\be\label{eq:nonsense}
T(x,d)\ll x/d^2+x^2/(d^2Q^{1+\varpi})+Q/d+x/(dQ^\varpi),
\ee
which we also use to replace~\eqref{eq:large d} in the previous case.
Optimizing the choice $Q$ in~\eqref{eq:nonsense}, leads to
$$
T(x,d)\ll xd^{-2}+x^{c_1}d^{-1}+x^{2c_2}d^{-1-c_2},
$$
where 
$$
c_1\defeq(1+\varpi)^{-1} \mand c_2\defeq(2+\varpi)^{-1}.
$$
Summing over $d>D$, we find that
\be
\label{eq: UB with E0}
\sum_{D<d\le x}\mu(d)\cdot
\bigl|\{n\le x/d:~\bnu_{d,n}\in\Omega_d\}\bigr|
\ll\sum_{D<d\le x}T(x,d)\ll E_0,
\ee
where
\be\label{eq:E0 def}
E_0\defeq xD^{-1}+x^{c_1}\log x+x^{2c_2}D^{-c_2}.
\ee 

\subsection{Small $d$}
\label{sec:small d ONE}

Next, we consider the contribution to
\eqref{eq:birdsong} from integers $d\le D$.
Using the definitions of~\S\ref{sec:discrepancy} and
the fact that $m(\Omega_d)=d^{-k}$, we have
\be\label{eq:Asymp}
\bigl|\bigl\{n\le x/d:~\bnu_{d,n}\in\Omega_d\bigr\}\bigr|
=A(\bnu_d,\Omega_d;x/d)
=xd^{-k-1}+O(xd^{-1}\sD_d),
\ee
where $\sD_d$ is shorthand for the discrepancy $\sD(\bnu_d;x/d)$.
By Lemma~\ref{lem:K-S}, for any positive real parameter $H\le x$ we have
$$
\sD_d\ll\frac{1}{H}
+\frac{d}{x}\sum_{0<\|\vh\|\le H}\frac{1}{r(\vh)}
\left|\sum_{n\le x/d}\e\bigl(h_kd^{m_k-1}\alpha_kn^{m_k}
+\cdots+ h_1d^{m_1-1}\alpha_1n^{m_1}\bigr)\right|,
$$
where the outer sum runs over all 
$\vh= (h_1,\ldots,h_k)\in\Z^k$ with $0<\|\vh\|\le H$.
For each $j=1,\ldots,k$, let $\cH_j$ be
the set of such vectors $\vh= (h_1,\ldots,h_k)$ with $h_j\ne 0$ and
$h_{j+1}=\cdots=h_k=0$. Then 
\be\label{eq:DdDdj}
\sD_d\ll\frac{1}{H}+\sum_{j=1}^k\sD_{d,j},
\ee
where
\be\label{eq:dentalfloss}
\sD_{d,j}\defeq (x/d)^{-1}\sum_{\vh\in\cH_j}\frac{1}{r(\vh)}
\bigl|S_j(d,\vh)\bigr|
\ee
and
$$
S_j(d,\vh)\defeq\sum_{n\le x/d}
\e\(h_jd^{m_j-1}\alpha_jn^{m_j}
+\cdots+h_1d^{m_1-1}\alpha_1n^{m_1}\).
$$ 
As in~\S\ref{sec:Large d ONE} we fix $\varpi\in(0,\tau^{-1})$.
For each $j=1,\ldots,k$, Lemma~\ref{lem:Dioph-one}
shows that for all large $Q$ (depending on $\alpha_j$ and $\varpi$)
there are integers $a$ and $q$ such that
\be\label{eq:a/q-alpha_j}
\biggl|\alpha_j-\frac{a}{q}\biggr|<\frac{1}{qQ},
\qquad\gcd(a,q)=1,\qquad Q^\varpi<q\le Q.
\ee 
As before, we assume that $Q\le x^{O(1)}$.

We now turn to the problem of bounding $\sD_{d,j}$ for any given $j$.
Because the different bounds on Weyl sums given in
\S\ref{sec:bounding-Weyl-sums} vary in strength,
we examine several cases according to whether 
$m_j=1$, $m_j=2$, or $m_j\ge 3$.

\begin{lemma}\label{lem:Ddj-bounds}
With the notation as above, we have for each $j$:
$$
\sD_{d,j}\ll\begin{cases}
x^{c_1-1}d^{1-c_1}H^{c_1}(\log H)^{1-c_1}
&\quad\hbox{if $m_j=1$},\\
x^{c_1-1}  d^{1-c_1/2}+x^{-1/2}d
&\quad\hbox{if $m_j=2$},\\
\displaystyle x^{o(1)}\biggl(\biggl(\frac{H^{c_1}d^{(m_j-1)c_1}}
{x^{(m_j-1)(1-c_1)}}\biggr)^{\lambda_j}
+\biggl(\frac{d}{x}\biggr)^{\lambda_j}\biggr)
&\quad\hbox{if $m_j\ge 3$},\\
\end{cases}
$$
where
\be\label{eq:c1 def}
c_1\defeq (1+\varpi)^{-1}
\mand
\lambda_j\defeq(m_j^2-m_j)^{-1}.
\ee
\end{lemma}

\begin{proof}
First, suppose that $m_j=1$.
For each vector $\vh=(h_1,0,\ldots,0)\in\cH_1$ we
apply Lemma~\ref{lem:LinSum-h} with
$$
N\defeq \fl{x/d} \mand h\defeq h_1, 
$$
deriving the bound
$$
S_1(d,\vh)=\sum_{n\le x/d}\e(h_1\alpha_1n)
\ll \frac{x|h_1|}{dq}+q
<\frac{x|h_1|}{dQ^\varpi}+Q,
$$
where we used~\eqref{eq:a/q-alpha_j} in the second step.
By~\eqref{eq:dentalfloss} we have
$$
\sD_{d,1}=\frac{d}{x}\sum_{\vh\in\cH_1}\frac{1}{|h_1|}
\(\frac{x|h_1|}{dQ^\varpi}+Q\)
\ll\frac{H}{Q^\varpi} +\frac{dQ\log H}{x}.
$$
Taking 
$$
Q\defeq \(\frac{xH}{d\log H}\)^{c_1},
$$ 
we obtain the desired bound
for $\sD_{d,1}$. 

Next, suppose that $m_j=2$.
For any vector $\vh=(h_1,h_2,0,\ldots,0)\in\cH_2$ we
apply Lemma~\ref{lem:QuadrSum-h}
with $N\defeq\fl{x/d}$ and $h\defeq h_2d$, deriving the bound
$$
S_2(d,\vh)
\ll\(\sum_{v=1}^N\min\(N,\frac{1}{\nearint{2h_2dv\alpha_2}}\)\)^{1/2},
$$
hence by~\eqref{eq:dentalfloss} (and symmetry) we have
\begin{align*}
\sD_{d,2}
&\ll\frac{d}{x}\sum_{h_1\le H}
\sum_{0<h_2\le H}\frac{1}{\max\(h_1,1\)h_2}
\(\sum_{v=1}^N\min\(N,
\frac{1}{\llbracket 2h_2dv\alpha_2\rrbracket}\)\)^{1/2}\\
&\ll\frac{d\log H}{x}\sum_{0<h\le H}\frac{1}{h}
\(\sum_{v=1}^N\min\(N,
\frac{1}{\llbracket 2hdv\alpha_2\rrbracket}\)\)^{1/2}.
\end{align*}
Splitting the summation range over $h$ into $O(\log H)$ dyadic 
intervals of the form $R<h\le 2R$ with $\tfrac12\le R\ll H$, 
it suffices to bound each term
\be
\label{eq:DR} 
\Sigma_R \defeq \frac{d\log H}{xR} \sum_{R<h\le 2R}
\(\sum_{v=1}^N\min\(N,\frac{1}{\llbracket 2hd v\alpha_2\rrbracket}\)\)^{1/2}.
\ee
By the Cauchy inequality,
\be
\label{eq:DR-FR} 
 \sum_{R<h\le 2R}
\(\sum_{v=1}^N\min\(N,\frac{1}{\llbracket 2hd v\alpha_2\rrbracket}\)\)^{1/2}
\le R^{1/2}\,\Xi_R^{1/2},
\ee 
where 
$$
\Xi_R\defeq \sum_{R<h\le 2R}
\sum_{v=1}^N\min\(N,\frac{1}
{\llbracket 2hd v\alpha_2\rrbracket}\).
$$
Collecting together  products $2hdv$ with the same value 
$\nu\defeq 2hdv$, and using a well-known bound on the
divisor function (see, e.g.,~\cite[Equation~(1.81)]{IwKow}),
we have
$$
\Xi_R\le(dNR)^{o(1)}\sum_{1\le\nu\le 4dNR}
\min\(N,\frac{1}{\llbracket\nu\alpha_2\rrbracket}\).
$$
Finally, using Lemma~\ref{lem:RecipSum} and~\eqref{eq:a/q-alpha_j},
we conclude that
\begin{align*}
\Xi_R& \le x^{o(1)} \(N+q\log q\) \(dNR/q +1\)\\
& \le x^{o(1)} \(dN^2R/q + dNR + q\)\\
& \le x^{o(1)} \(\frac{dN^2R}{Q^\varpi} + dNR + Q\),
\end{align*}
which together with~\eqref{eq:DR} and~\eqref{eq:DR-FR} implies
$$
\Sigma_R\le x^{-1+o(1)}d\(\frac{dN^2}{Q^\varpi}+dN+QR^{-1}\)^{1/2}
\le x^{-1+o(1)}d\(\frac{x^2}{dQ^\varpi}+x+Q\)^{1/2}.
$$
The bound is optimized with the choice $Q\defeq x^{2c_1} d^{-c_1}$,
and we get that
$$
\Sigma_R\le x^{o(1)}\(x^{c_1-1}  d^{1-c_1/2}+x^{-1/2}d\).
$$
Summing over all possibilities for $R$, we finish the proof in this case.

Finally, suppose that $m_j\ge 3$. For each vector $\vh\in\cH_j$
we use the bound~\eqref{eq:Bound little o} from Lemma~\ref{lem:Weyl-h}
with $N\defeq \fl{x/d}$ and $h\defeq h_jd^{m_j-1}$; taking into
account~\eqref{eq:a/q-alpha_j}, we find that the bound
$$
S_j(d,\vh)\ll N^{1+o(1)}\(\frac{|h_jd^{m_j-1}}{Q^\varpi}+\frac{1}{N}
+\frac{Q}{N^{m_j-1}}\)^{\lambda_j} 
$$ 
holds with $\lambda_j\defeq (m_j^2-m_j)^{-1}$ as in~\eqref{eq:c1 def}.
To optimize the bound, we choose
$Q\defeq (h_jd^{m_j-1}N^{m_j-1})^{c_1}$, which leads to
$$
S_j(d,\vh)\ll N^{1+o(1)}\(\frac{|h_j|^{c_1}d^{(m_j-1)c_1}}
{N^{(m_j-1)(1-c_1)}}+\frac{1}{N}\)^{\lambda_j}.
$$
Recalling~\eqref{eq:dentalfloss} and noting that (with any fixed $C>0$)
\be\label{eq:Hsums}
\sum_{\vh\in\cH_j}\frac{1}{r(\vh)}\ll\(\log H\)^{j}
\quad\text{and}\quad
\sum_{\vh\in\cH_j}\frac{|h_j|^C}{r(\vh)}\ll H^C\(\log H\)^{j-1}, 
\ee
we derive the bound
\begin{align*}
\sD_{d,j}
&\le x^{o(1)}\(\frac{H^{c_1}d^{(m_j-1)c_1}}
{N^{(m_j-1)(1-c_1)}}+\frac{1}{N}\)^{\lambda_j}\\
&=x^{o(1)}\(\(\frac{H^{c_1}d^{(m_j-1)c_1}}
{x^{(m_j-1)(1-c_1)}}\)^{\lambda_j}+\(\frac{d}{x}\)^{\lambda_j}\).
\end{align*} 
which concludes the proof.
\end{proof}

We are now in a position to bound $\sD_d$ and to estimate the 
overall contribution to~\eqref{eq:birdsong} from integers $d\le D$.
We consider the cases $k=1$ and $k\ge 2$ separately.

\bigskip\noindent{\sc Case~1: $(k=1)$.}
In this case, $m_k=1$.
By~\eqref{eq:DdDdj} and Lemma~\ref{lem:Ddj-bounds} we have
$$
\sD_d\ll H^{-1}+x^{c_1-1}d^{1-c_1} H^{c_1}(\log H)^{1-c_1}.
$$
The bound is optimized with the choice
$H\defeq\(x/\(d\log x\)\)^{1-2c_2}$, where
$$
c_2\defeq (2+\varpi)^{-1},
$$
and with this choice, we get that
$$
\sD_d\ll   \(\frac{d\log x}{x}\)^{1-2c_2}. 
$$
Using this result in~\eqref{eq:Asymp} and summing over all $d\le D$,
it follows that
\be\label{eq: Asym with E1}
\sum_{d\le D}\mu(d)\cdot
\bigl|\{n\le x/d:~\bnu_{d,n}\in\Omega_d\}\bigr|
=\frac{x}{\zeta(2)}+O(E_1),
\ee
where
\be\label{eq:E1 def}
E_1\defeq xD^{-1}+ x^{2c_2}(\log x)^{1-2c_2}D^{1-2c_2}. 
\ee

\bigskip\noindent{\sc Case~2: $(k\ge 2)$.} 
By~\eqref{eq:DdDdj}, and putting together all available estimates from 
Lemma~\ref{lem:Ddj-bounds}, we obtain the bound
$$
\sD_d\ll  \(H^{-1}+x^{c_1-1}d^{1-c_1}H^{c_1} + x^{c_1-1}  d^{1-c_1/2}+x^{-1/2}d+\sum_{j=2}^k\sD_{d,j} \)x^{o(1)},
$$
where  for $2\le j\le k$ we have
$$
\sD_{d,j}\le x^{o(1)}\biggl(\biggl(\frac{H^{c_1}d^{(m_j-1)c_1}}
{x^{(m_j-1)(1-c_1)}}\biggr)^{\lambda_j}
+\biggl(\frac{d}{x}\biggr)^{\lambda_j}\biggr).
$$
Clearly, the term $(d/x)^{\lambda_j}$ increases with
the parameter $j$.  
On the other hand, applying Lemma~\ref{lem:elem} with 
$u\defeq H^{c_1}$ and $v\defeq x^{1-c_1}d^{-c_1}$,
we see that
the first term does not decrease with $j$ 
provided that $u\le v$, or equivalently, when
$$
H\le x^{(1-c_1)/c_1} d^{-1} = x^\varpi d^{-1}.
$$ 
Assuming this condition is met, we derive the bound
$$
\sD_d\ll\(\frac{1}{H}+\frac{d^{1-c_1}H^{c_1}}{x^{1-c_1}}
+\frac{d^{1-c_1/2}}{x^{1-c_1}}+\frac{d}{x^{1/2}}+
\frac{H^{c_3}d^{c_4}}{x^{c_5}}
+\frac{d^\lambda}{x^\lambda}\)x^{o(1)},
$$
where
\be\label{eq:mlambda-defns}
m\defeq m_k,
\qquad \lambda\defeq (m^2-m)^{-1},
\ee
and 
\be\label{eq:c345-defns}
c_3\defeq \lambda c_1,
\qquad c_4\defeq \lambda(m-1)c_1,
\qquad c_5\defeq \lambda(m-1)(1-c_1).
\ee
It is now easy to see that there exists a choice of the
parameter $H\in[0,x^\varpi d^{-1}]$,  for which
$$
\frac{1}{H}+\frac{d^{1-c_1}H^{c_1}}{x^{1-c_1}}+\frac{H^{c_3}d^{c_4}}{x^{c_5}}
\ll \frac{d}{x^\varpi} 
+\frac{d^{(1-c_1)/(c_1+1)}}{x^{(1-c_1)/(c_1+1)}}
+\frac{d^{c_4/(c_3+1)}}{x^{c_5/(c_3+1)}}. 
$$
Hence
$$
\sD_d\ll\(
\frac{d}{x^\varpi}
+\frac{d^\lambda}{x^\lambda}
+\frac{d}{x^{1/2}}
+\frac{d^{1-c_1/2}}{x^{1-c_1}}
+\frac{d^{(1-c_1)/(c_1+1)}}{x^{(1-c_1)/(c_1+1)}}
+\frac{d^{c_4/(c_3+1)}}{x^{c_5/(c_3+1)}}
\)x^{o(1)}.
$$
Using this result in~\eqref{eq:Asymp} and summing over all $d\le D$,
we have
\be
\label{eq: Asym with E2}
\sum_{d\le D}\mu(d)\cdot
\bigl|\{n\le x/d:~\bnu_{d,n}\in\Omega_d\}\bigr|
=\frac{x}{\zeta(k+1)}+O(E_2),
\ee
where
\be\label{eq:E2 def}
\begin{split} 
E_2\defeq
\frac{x}{D^k}+x^{1+o(1)}\Bigl(
\frac{D}{x^\varpi} 
+\frac{D^\lambda}{x^\lambda}
&+\frac{D}{x^{1/2}}
+\frac{D^{1-c_1/2}}{x^{1-c_1}}\\
& \quad +
\frac{D^{(1-c_1)/(c_1+1)}}{x^{(1-c_1)/(c_1+1)}}
+\frac{D^{c_4/(c_3+1)}}{x^{c_5/(c_3+1)}}\Bigr).
\end{split} 
\ee

\subsection{Final optimizations}
When $m_k = 1$, that is, in 
 {\sc Case~1}, we combine~\eqref{eq: UB with E0}, \eqref{eq:E0 def},
\eqref{eq: Asym with E1}, and~\eqref{eq:E1 def},
obtaining an overall error
$$
E\defeq E_0+E_1\ll x^{o(1)}\(xD^{-1}+x^{c_1}+x^{2c_2}D^{1-c_2}\).
$$
Recalling the definitions of $c_1$ and $c_2$, and taking
$D\defeq x^{\varpi/(3+2\varpi)}$, one has
$$
E\ll x^{o(1)}\(x^{(3+\varpi)/(3+2\varpi)}+x^{1/(1+\varpi)}\)
\qquad(x\to\infty).
$$
Since
$$
\frac{3+\varpi}{3+2\varpi}>\frac{1}{1+\varpi}
\qquad(0<\varpi<1),
$$ 
the second term can be dropped, and thus
$$
E\ll x^{(3+\varpi)/(3+2\varpi)+o(1)}\qquad(x\to\infty).
$$
Letting $\varpi$ approach $\tau^{-1}$,
Theorem~\ref{thm:finite-type} follows in this case.

When $m_k \ge  2$, that is, in 
{\sc Case~2}, we use~\eqref{eq: UB with E0},  \eqref{eq:E0 def},
\eqref{eq: Asym with E2},
and~\eqref{eq:E2 def}, observing also that the term $x^{c_1}\log x$ from~\eqref{eq:E0 def}
can be discarded since it is dominated by the term
$D^{1-c_1/2}x^{c_1}$ in~\eqref{eq:E2 def}.  
Hence, the overall error becomes 
\be\label{eq:E messy}
\begin{split} 
E\defeq E_0+E_2\ll 
x^{1+o(1)}\biggl(\frac{1}{D}&
+\frac{1}{x^{1-2c_2}D^{c_2}}+\frac{D}{x^\varpi}
+\frac{D^\lambda}{x^\lambda}
+\frac{D}{x^{1/2}}
\\
&+\frac{D^{1-c_1/2}}{x^{1-c_1}} + 
\frac{D^{(1-c_1)/(c_1+1)}}{x^{(1-c_1)/(c_1+1)}}
+\frac{D^{c_4/(c_3+1)}}{x^{c_5/(c_3+1)}}\biggr).
\end{split} 
\ee

We now choose
$$
D\defeq x^{\varpi/8}.
$$
Since 
$$
D=x^{\varpi/8}<x^{1/8} \mand  1-c_1\in(\tfrac12\varpi,\varpi),
$$ 
we have the bound
$$
\frac{1}{D}+\frac{D}{x^\varpi}
+\frac{D^\lambda}{x^\lambda}+\frac{D}{x^{1/2}} 
\le x^{-\varpi/8} + x^{-7\varpi/8}+ x^{-7\lambda/8} +x^{-3/8},
$$
which we rewrite crudely in the form
\be\label{eq:Bound 1}
\frac{1}{D}+\frac{D}{x^\varpi}
+\frac{D^\lambda}{x^\lambda} +\frac{D}{x^{1/2}} 
\ll   x^{-\frac{1}{8}\min\{\varpi,\lambda\}}. 
\ee

Next, using 
$$
1-2c_2=1-\frac{2}{2+\varpi}=\frac{\varpi}{2+\varpi}\ge\varpi/3,
$$
we estimate 
$$
%%\be
%%\label{eq:Bound 2}
\frac{1}{x^{1-2c_2}D^{c_2}}\ll x^{-\varpi/3}.
%%\ee
$$
Similarly, we have
$$
%%\be\label{eq:Bound 3}
\frac{D^{1-c_1/2}}{x^{1-c_1}}\ll  
\frac{D}{x^{\varpi/2}}\le x^{-3\varpi/8}
%%\ee
$$
and
$$
%%\be
%%\label{eq:Bound 4}
\frac{D^{1/(c_1+1)}}{x^{(1-c_1)/(c_1+1)}} \le 
\(\frac{D}{x^{1-c_1}}\)^{1/2}\ll  x^{-3\varpi/16}.
%%\ee
$$

Finally, recalling~\eqref{eq:mlambda-defns}
and~\eqref{eq:c345-defns}, we see that for $m \ge 2$
$$
c_3<1,\qquad
c_4\le\lambda m=\frac{1}{m-1}\le\frac{2}{m},\qquad
c_5=\frac{1-c_1}{m}=\frac{\varpi}{m(1+\varpi)}\ge\frac{\varpi}{2m},
$$
and therefore
\be\label{eq:Bound 5}
\frac{D^{c_4/(c_3+1)}}{x^{c_5/(c_3+1)}}
\le\(\frac{D^{c_4}}{x^{c_5}}\)^{1/2}
\le\(\frac{D^{2/m}}{x^{\varpi/(2m)}}\)^{1/2}
=x^{-\varpi/(8m)}.
\ee

Collecting the bounds~\eqref{eq:Bound 1}$-$\eqref{eq:Bound 5}
into \eqref{eq:E messy}, we find that 
$$
E\ll x^{1-\frac{1}{8}\min\{\varpi/m,\lambda\}+o(1)}\qquad(x\to\infty).
$$
Letting $\varpi$ approach $\tau^{-1}$,
Theorem~\ref{thm:finite-type} follows in this case,
and we are done.

\section{Proof of Theorem~\ref{thm:fin-exp-type}}
\label{sec:proof-thm2}

\subsection{Plan of the proof}
We follow a plan similar to the one outlined in \S\ref{sec:plan}.
We proceed as in the proof of Theorem~\ref{thm:finite-type},
however we now use Lemma~\ref{lem:Dioph-two} instead of
Lemma~\ref{lem:Dioph-one} and the bound~\eqref{eq:Bound log}
instead of~\eqref{eq:Bound little o}. We continue to use
the notation introduced earlier, and since our arguments are
essentially the same, we focus only on the needed adjustments.
 
Note that we can assume $\tau_\star<(m_k^2-m_k+1)^{-1}$
since the statement of the theorem is trivial otherwise. 

\subsection{Large $d$}

Let $\varpi\in(0,\tau_\star^{-1}-1)$ be fixed in what follows.
As before, we start by considering the contribution to
\eqref{eq:birdsong} coming from integers $d >D$, where $D$
is a real parameter to be specified below.

Lemma~\ref{lem:Dioph-two} shows that
for all large $Q$ (depending on $\alpha_1$
and $\varpi$) there are integers $a$ and $q$ such that
$$
\biggl|\alpha_1-\frac{a}{q}\biggr|<\frac{1}{qQ},
\qquad\gcd(a,q)=1,\qquad (\log Q)^{\varpi+1}<q\le Q.
$$
Using this result in place of~\eqref{eq:a/q-alpha_j},
the argument of~\S\ref{sec:Large d ONE} yields the bound
$$
T(x,d)\ll x/d^2+x^2/(d^2Q(\log Q)^{\varpi+1})+Q/d+x/(d(\log Q)^{\varpi+1})
$$
instead of~\eqref{eq:nonsense}. Taking $Q\defeq x/(\log x)^{\varpi+1}$,
it follows that
$$
T(x,d)\ll x/d^2+x/(d(\log x)^{\varpi+1}).
$$
and therefore
\be\label{eq:E0 def TWO}
\sum_{D<d\le x}\mu(d)\cdot
\bigl|\{n\le x/d:~\bnu_{d,n}\in\Omega_d\}\bigr|
\ll E_0\defeq xD^{-1}+x(\log x)^{-\varpi}.
\ee
Below, we use this bound in place
of~\eqref{eq: UB with E0} and~\eqref{eq:E0 def}.

\subsection{Small $d$}
\label{sec:small d TWO}

Next, we consider the contribution to
\eqref{eq:birdsong} from integers $d\le D$.

As before, Lemma~\ref{lem:Dioph-two} shows that for all large $Q$
(depending only on $\balpha$ and~$\varpi$) and every $j=1,\ldots,k$,
there are integers $a$ and $q$ such that
\be\label{eq:a/q-alpha_j TWO}
\biggl|\alpha_j-\frac{a}{q}\biggr|<\frac{1}{qQ},
\qquad\gcd(a,q)=1,\qquad (\log Q)^{\varpi+1}<q\le Q.
\ee

\begin{lemma}\label{lem:Ddj-bounds TWO}
In the notation of~\S\ref{sec:small d ONE}, we have for each $j$:
$$
\sD_{d,j}\ll\begin{cases}
H(\log x)^{-\varpi-1}
&\quad\hbox{if $m_j=1$},\\
H^{\lambda_j}(\log H)^{j-1} d^{(m_j-1)\lambda_j}(\log x)^{1-(\varpi+1)\lambda_j}
&\quad\hbox{if $m_j\ge 2$},\\
\end{cases}
$$
where
\be\label{eq:c1 def TWO}
\lambda_j\defeq(m_j^2-m_j+2)^{-1}.
\ee
\end{lemma}

\begin{proof}
First, suppose that $m_j=1$.
For each vector $\vh=(h_1,0,\ldots,0)\in\cH_1$ we
apply Lemma~\ref{lem:LinSum-h} with
$N\defeq \fl{x/d}$ and $h\defeq h_1$, deriving the bound
$$
S_1(d,\vh)=\sum_{n\le x/d}\e(h_1\alpha_1n)\ll \frac{x|h_1|}{dq}+q
<\frac{x|h_1|}{d(\log Q)^{\varpi+1}}+Q,
$$
where we have used~\eqref{eq:a/q-alpha_j TWO} in the second step.
By~\eqref{eq:dentalfloss} we have
$$
\sD_{d,1}=\frac{d}{x}\sum_{\vh\in\cH_1}\frac{1}{|h_1|}
\(\frac{x|h_1|}{d(\log Q)^{\varpi+1}}+Q\)
\ll\frac{H}{(\log Q)^{\varpi+1}}+\frac{dQ\log H}{x}.
$$
Taking $Q\defeq xH/(d(\log x)^{\varpi+2})$ we obtain the bound
for $\sD_{d,1}$ stated in the lemma.

Next, suppose that $m_j\ge 2$. For each vector $\vh\in\cH_j$
we use the bound~\eqref{eq:Bound log} of Lemma~\ref{lem:Weyl-h}
with $N\defeq \fl{x/d}$ and $h\defeq h_jd^{m_j-1}$; taking into
account~\eqref{eq:a/q-alpha_j TWO} and using the crude inequality $\gcd(q,h)\le Q$, we find that the bound
$$
S_j(d,\vh)\ll \frac{x\log x}{d}
\(\frac{|h_j| d^{m_j-1}}{(\log Q)^{\varpi+1}}+\frac{d}{x}
+\frac{Qd_j^{m_j-1}}{x^{m_j-1}}\)^{\lambda_j} 
$$ 
holds with $\lambda_j$ as in~\eqref{eq:c1 def TWO}. To optimize, we choose
$$
Q\defeq \frac{|h_j|x^{m_j-1}}{(\log x)^{\varpi+1}},
$$
which leads to the bound
\begin{align*} 
S_j(d,\vh) & \ll \frac{x\log x}{d}
\(\frac{|h_j|d^{m_j-1}}{(\log x)^{\varpi+1}}+\frac{d}{x}\)^{\lambda_j}\\
& \ll  \frac{x\log x}{d}
\(\frac{|h_j|d^{m_j-1}}{(\log x)^{\varpi+1}} \)^{\lambda_j}
=\frac{x|h_j|^{\lambda_j} d^{(m_j-1)\lambda_j-1}}{(\log x)^{(\varpi+1)\lambda_j-1}}.
\end{align*}
Using~\eqref{eq:dentalfloss} and~\eqref{eq:Hsums},
we derive the bound for $\sD_{d,j}$ stated in the lemma.
\end{proof}

We now bound $\sD_d$ and estimate the overall
contribution to~\eqref{eq:birdsong} coming from integers $d\le D$,
considering separately the cases $k=1$ and $k\ge 2$. 

\bigskip\noindent{\sc Case~1: $(k=1)$.} In this case $m_j = m_1=1$. By~\eqref{eq:DdDdj} and Lemma~\ref{lem:Ddj-bounds TWO} we have
$$
 \sD_d\ll H^{-1}+ H(\log x)^{-\varpi-1}.
$$
The bound is optimized with the choice
$H\defeq (\log x)^{(\varpi+1)/2}$, which gives
$$
\sD_d\ll(\log x)^{-(\varpi+1)/2}.
$$
Using this result in~\eqref{eq:Asymp} and summing over all $d\le D$,
it follows that
\be 
\label{eq: Asym with E1 TWO}
\sum_{d\le D}\mu(d)\cdot
\bigl|\{n\le x/d:~\bnu_{d,n}\in\Omega_d\}\bigr|
=\frac{x}{\zeta(2)}+O(E_1),
\ee
where
\be\label{eq:E1 def TWO}
E_1\defeq xD^{-1}+x(\log x)^{-(\varpi+1)/2}\log D. 
\ee

\bigskip\noindent{\sc Case~2: $(k\ge 2)$.}
By~\eqref{eq:DdDdj} and Lemma~\ref{lem:Ddj-bounds TWO} we have
$$
\sD_d\ll \frac{1}{H}+\frac{H}{(\log x)^{\varpi+1}}
+(\log x)(\log H)^{k-1}\sum_{j=2}^k
\(\frac{Hd^{m_j-1}}{(\log x)^{\varpi+1}}\)^{\lambda_j},
$$
where $\lambda_j\defeq(m_j^2-m_j+2)^{-1}$. Using Lemma~\ref{lem:elem TWO},
we see that the terms in the above sum are nondecreasing as $j$ increases
provided that
\be\label{eq:Dcond TWO}
 d^{m_k-1}\le\frac{(\log x)^{\varpi+1}}{H}.
\ee
Assuming this for the moment, we have
\be\label{eq:marvelous}
\sD_d\ll \frac{1}{H}+\frac{H}{(\log x)^{\varpi+1}}
+ \frac{H^\lambda(\log H)^{k-1}d^{(m-1)\lambda}}{(\log x)^{(\varpi+1)\lambda-1}},
\ee
where
$$
m\defeq m_k\mand
\qquad \lambda\defeq (m^2-m+2)^{-1}.
$$
Clearly,~\eqref{eq:Dcond TWO} implies  
$H\le (\log x)^{\varpi+1}$, so we can drop
the second term in the bound~\eqref{eq:marvelous} since it is always 
dominated by the third term. We set
$$
H\defeq \(\frac{(\log x)^{(\varpi+1)\lambda-1}}{d^{(m-1)\lambda}}\)^{1/(\lambda+1)}
$$
to balance the two remaining terms in~\eqref{eq:marvelous}.
Note that
$$
d^{(m-1)\lambda}H^\lambda=\frac{(\log x)^{(\varpi+1)\lambda}}{H\log x}\le
(\log x)^{(\varpi+1)\lambda},
$$
and therefore the condition~\eqref{eq:Dcond TWO} is met.
Putting everything together, we get that
$$
\sD_d\ll(\log x)^{o(1)}
\(\frac{d^{(m-1)\lambda}}{(\log x)^{(\varpi+1)\lambda-1}}\)^{1/(\lambda+1)}
\qquad(x\to\infty).
$$
Inserting this bound into~\eqref{eq:Asymp}
and summing over $d\le D$, we find that
\be\label{eq: Asym with E2 TWO}
\sum_{d\le D}\mu(d)\cdot
\bigl|\{n\le x/d:~\bnu_{d,n}\in\Omega_d\}\bigr|
=\frac{x}{\zeta(k+1)}+ O(E_2),
\ee
where
\be\label{eq:E2 def TWO}
E_2=xD^{-k}+x(\log x)^{o(1)}
\(\frac{D^{(m-1)\lambda}}{(\log x)^{(\varpi+1)\lambda-1}}\)^{1/(\lambda+1)}
\qquad(x\to\infty).
\ee

\subsection{Final optimizations}
In {\sc Case~1}, we combine~\eqref{eq:E0 def TWO},
\eqref{eq: Asym with E1 TWO}, and~\eqref{eq:E1 def TWO}, 
which yields an overall error
$$
E\defeq E_0+E_1\ll xD^{-1}+x(\log x)^{-\varpi}+x(\log x)^{-(\varpi+1)/2}\log D
$$
Choosing $D\defeq (\log x)^{\varpi}$ in this case, we have
$$
E\ll x(\log x)^{-\varpi} + x(\log x)^{-(\varpi+1)/2+o(1)}\qquad(x\to\infty).
$$
Letting $\varpi$ approach $\tau_\star^{-1}$, we finish the proof of
Theorem~\ref{thm:fin-exp-type} in this case.

In {\sc Case~2}, we combine~\eqref{eq:E0 def TWO},
\eqref{eq: Asym with E2 TWO}, and~\eqref{eq:E2 def TWO}, 
which yields an overall error
$$
E\defeq E_0+E_2\ll xD^{-1}+ x(\log x)^{-\varpi}+x(\log x)^{o(1)}
\(\frac{D^{(m-1)\lambda}}{(\log x)^{(\varpi+1)\lambda-1}}\)^{1/(\lambda+1)}
$$
as $x\to\infty$. We balance this bound by taking
$$
D\defeq (\log x)^\vartheta,\qquad
\vartheta\defeq\frac{(\varpi+1)\lambda-1}{m\lambda+1},
$$
and with this choice, we can drop the middle term $x(\log x)^{-\varpi}$
since $\vartheta<\varpi$.
Letting $\varpi$ approach $\tau_\star^{-1}$, we finish the proof of
Theorem~\ref{thm:fin-exp-type}.

\section*{Acknowledgements}

The authors are grateful to the organizers of the
{\it Number Theory Web Seminar\/}, \url{https://www.ntwebseminar.org},
whose vision, hard work, and dedication helped to keep the spirit
of mathematical exchange alive and vibrant during the midst of a
pandemic. The present work was motivated by a captivating talk given
by Robert Tichy at this seminar.

We also thank Vitaly Bergelson and Florian Richter for their comments
and an anonymous referee for very important suggestions and for
pointing out some gaps in our initial arguments.

During the preparation of this paper,
I.~E.~Shparlinski was supported in part by the Australian Research Council 
Grants DP230100530 and DP230100534.

\end{document}